\documentclass[a4paper, 11pt]{article} 

\usepackage[utf8]{inputenc}
\usepackage{lmodern}
\usepackage{geometry}
\usepackage{amssymb, amsfonts, amsthm, amsmath, mathtools}
\usepackage[english]{babel}
\usepackage[colorlinks]{hyperref}
\usepackage{tikz}
\theoremstyle{plain}
\newtheorem{theorem}{Theorem}[section]

\newtheorem{lemma}[theorem]{Lemma}
\newtheorem{proposition}[theorem]{Proposition}

\theoremstyle{definition}

\theoremstyle{remark}

\numberwithin{equation}{section}

\renewcommand{\bar}[1]{\overline{#1}}

\renewcommand{\tilde}[1]{\widetilde{#1}}

\renewcommand{\rho}{\varrho}
\renewcommand{\epsilon}{\varepsilon}

\usepackage{xcolor}

\title{The minimal quasi-stationary distribution of the absorbed $M/M/\infty$ queue}
\author{Elie~Cerf\footnote{Université Sorbonne Paris Nord, LAGA, CNRS, UMR 7539, Laboratoire d'excellence Inflamex, F-93430 Villetaneuse, France.
		cerf@math.univ-paris13.fr} }
\date{July 13, 2023}

\begin{document}

\maketitle

\begin{abstract}
	 We give in this paper two characterizations of the minimal exponential rate of survival $\theta^*$ of the $M/M/\infty$ queue. The first one is a derivation of a known result on the duration of excursions of this process. The second one was conjectured by Mart\'{\i}nez and Ycart \cite{MartinezYcart_cutoff} and is proved using complex analysis by establishing a connection with the first characterization in terms of the incomplete gamma function.
\end{abstract}

\section{Introduction}

Throughout the end of the last century, quasi-stationary distributions have emerged as a powerful tool to study the long time behavior of absorbed Markov processes. Let us consider a Markov process $(X_t)_{t\geq0}$ evolving on a state space $E=\tilde{E}\cup \Delta$, where $\Delta$ is the set of absorbing states and assume that the process is almost surely eventually absorbed, meaning that if we denote by $T := \inf\{s\geq 0, X_s \in \Delta \}$ the time of absorption, then for any initial distribution $\mu$ \[\mathbb{P}_\mu(T<\infty)=1.\] In this setting, $(X_{t\wedge T})_{t\geq0}$ admits only trivial stationary distributions concentrated on $\Delta$. Therefore we define its quasi-stationary distribution as follows. Let $\nu$ be a probability measure on $\tilde{E}$, it is said to be a quasi-stationary distribution of $(X_t)_{t\geq0}$ if $\nu$ verifies for all $i\in \tilde{E}$ : 

\begin{equation}\label{eqn:defqsd}
	\mathbb{P}_\nu (X_t=i | t<T ) = \nu(i). 
\end{equation}

In fact, like stationary measures, quasi-stationary distributions can be linked to the limiting behavior of the process $(X_t)_{t\geq0}$, more precisely, one can show that, when it exists for any $j\in \tilde{E}$, the following Yaglom limit does not depend on $j$ and uniquely defines a quasi-stationary distribution called the minimal quasi-stationary distribution:  

\begin{equation}\label{eqn:defnu*}
	\nu^*(i) = \underset{t\to\infty}{\lim}\mathbb{P}_j(X_t = i | t<T ), \mbox{ } \forall i\in \tilde{E}. 
\end{equation}

One can also remark that $\eqref{eqn:defqsd}$ implies that for any quasi-stationary distribution $\nu$, $T$ is exponentially distributed under distribution $\mathbb{P}_\nu$ \cite[Theorem 2.2. p.~19]{ColletMartinez_book}, meaning that if we denote by $\theta(\nu)$ its parameter, we have for all $t>0$ 
\begin{equation}\label{eqn:expkilling}
	\mathbb{P}_\nu(T>t) = e^{-\theta(\nu) t}.
\end{equation}
The parameter $\theta(\nu)$ is sometimes called the exponential rate of survival of $\nu$ and is maximal for the minimal quasi-stationary distribution $\nu^*$ \[\theta(\nu^*) = \sup\{\theta(\nu),\mbox{ } \nu \text{ is a quasi-stationary distribution of $X$}\}.\] Hence, studying the long-time behavior of an absorbed process $(X_t)_{t\geq 0}$ appears as strongly linked to the study of its minimal quasi-stationary distribution and of its parameter.  Moreover, $\eqref{eqn:expkilling}$ shows that a necessary condition for the existence of quasi-stationary distributions is the existence of an exponential moment for the time of absorption $T$ \cite{good_1968}, i.e 
\begin{equation}\label{eqn:deftheta*}
	\theta^* := \sup\{\theta>0, \exists i\in\mathbb{N}^* : \mathbb{E}_i[e^{\theta T}]<\infty \} > 0,
\end{equation}
and we call $\theta^*$ the exponential rate of survival of the process $(X_t)_{t\geq0}$. For a complete overview on quasi-stationary distribution, the reader can refer to the book by Collet et al. \cite{ColletMartinez_book}.

In this paper, we study the quasi-stationary distributions of a birth-and-death process on $\mathbb{N}$, with absorption at $0$. The applications of birth-and-death processes in biology to the study of the dynamics of populations have motivated several authors to look at the quasi-stationary distributions of such processes \cite{MeleardVillemonaisQSDpopulation}. In fact, it has been shown that the condition $\{\theta^*>0\}$ is not only necessary but sufficient for the existence of quasi-stationary distributions \cite{Ferrariexistenceqsd}. Moreover, in this case, for any given $\theta\in(0,\theta^*]$, one can construct a unique quasi-stationary distribution $\nu_\theta$ with exponential rate of survival $\theta$ of $(X_t)_{t\geq0}$ and vice-versa. And it is also the parameter of the minimal quasi-stationary distribution $\nu^*$ defined by $\eqref{eqn:defnu*}$, i.e \[ \theta^* = \theta(\nu^*). \] Therefore, it is crucial to be able to compute or at least approximate the minimal rate of survival $\theta^*$ if we want to understand the quasi-stationarity of a birth-and-death process. In some cases, authors have been able to give an explicit characterization of $\theta^*$, like for the symmetric random walk or the linear birth-and-death chain \cite{vanDoornBDP91}. But other more difficult or general cases are left open problems and still studied \cite{VanDoorn2006withKillling}, \cite{MeleardBDP2022}. The aim of this paper is to focus on the case of the $M/M/\infty$ queue.

Let us now formally introduce our process of interest. Let $(X_t)_{t\geq0}$ be a continuous-time birth-and-death Markov chain on $\mathbb{N}=\{0,1,2,\dots\}$, with $\Delta=\{0\}$ as an absorbant state, counting the number of customer in a queue system with rate matrix $Q:=(q_{i,j})_{i,j\in\mathbb{N}}$ given by : \[ q_{i,i+1} = a, \mbox{ } q_{i,i-1} = iq, \mbox{ } , \mbox{ } q_{i,i}=-(a+iq), \mbox{ for } i\in\mathbb{N^*}. \] In this process, a new customer arrives at rate $a>0$ while each customer is served at rate $q>0$. Thus, it is usually called the $M/M/\infty$ queue in relation to Kendall's notation in queueing theory : Makovian arrivals/Markovian service time/ infinite number of servers. These processes are of great interest in queueing theory and have been studied by computer scientists for a long time. In particular, in \cite{Guilleminsimonian} Guillemin and Simonian computed a formula for the Laplace transform of excursion times above any given level of the $M/M/\infty$ queue using Kummer functions, and in \cite{GuilleminPinchon98} Guillemin and Pinchon managed to get another, simpler formula for the same variable through the study of continued fractions linked to the process. 

In an article studying the decay rates of birth-and-death processes, Mart\'{\i}nez and Ycart found an upper bound for the exponential rate of survival $\theta^*$ of the $M/M/\infty$ queue and conjectured its exact value \cite{MartinezYcart_cutoff}. We validate this conjecture in this article, proving the following result : 

\begin{theorem}\label{thm:main}
	The exponential rate of survival $\theta^*$ of $(X_t)_{t\geq0}$ is the unique solution in $(0,q)$ of : 
	\begin{equation}\label{eqn:theta*isgamma*}
		1=\frac{a}{q}\int_{0}^{1}(1-x)^{-\frac{\theta}{q}}e^{-\frac{ax}{q}}dx.
	\end{equation}
\end{theorem}

In a first part, we use the work of Guillemin and Pinchon \cite{GuilleminPinchon98} to give a first characterization of $\theta^*$ using a special function. Then, by complex analysis arguments, we prove that this first characterization is equivalent to the one of Theorem \ref{thm:main}.

\section{Laplace transform of the time of absorption}

Our first characterization of the exponential rate of survival of $(X_t)_{t\geq0}$ is given by its definition $\eqref{eqn:deftheta*}$ as the first singularity point of the Laplace transform of the time of absorption of the process. Remark that, thanks to the irreducibility of $(X_t)_{t\geq0}$ the point of explosion of $\theta\mapsto\mathbb{E}_i[e^{\theta T }]$ does not depend on $i$, so we may only consider the trajectories of the process starting at $1$ and write 

\begin{equation}\label{eqn:theta*apartirde1}
	\theta^* = \sup\{\theta>0, \mathbb{E}_1[e^{\theta T }] < \infty\}.
\end{equation}

Guillemin and Pinchon \cite{GuilleminPinchon98} made an exhaustive study of the excursions of the non-killed $M/M/\infty$ queue. In particular, they were able to compute an explicit formula for the Laplace transform of the duration of an excursion by an $M/M/\infty$ queue above any given state $C\in\mathbb{N}$. We use this to compute the Laplace transform of $T$ as the length of an excursion above level $0$. 

\begin{proposition}[Laplace transform of $T$]\label{prop:LaplaceT}
	For any $s\in\mathbb{C}\setminus\{0, -1 , -2, \dots\}$ and $x\in\mathbb{C}$, define 
	\begin{equation}\label{eqn:bargamma}
		\bar{\gamma}(s,x) := \sum_{k=0}^{\infty} \frac{(-x)^k}{(k+s)k!}.
	\end{equation}
	The Laplace transform of $T = \inf\{t>0, X_t=0\}$ is given by :
	\begin{equation}\label{eqn:LaplaceT}
		\mathbb{E}_1[e^{\theta T }]= \frac{a-\theta}{a} - \frac{1}{\frac{a}{q}e^{-\frac{a}{q}}\bar{\gamma}(-\frac{\theta}{q},-\frac{a}{q})}, \mbox{ for any } \theta>0.
	\end{equation}
\end{proposition}

\begin{proof}
	Guillemin and Pinchon only give the formula when the serving rate is $q=1$: consider a $M/M/\infty$ queue $(Y_t)_{t\geq0}$ with arrival rate $\tilde{a}$ and serving rate $1$ and denote $\tilde{T} = \inf\{t>0, Y_t=0\}$, then we have for any $\theta>0$ 
	\begin{equation}\label{eqn:laplaceavantchangementtemps}
		\mathbb{E}_1[e^{\theta \tilde{T}}] = \frac{\tilde{a} - \theta}{\tilde{a}} - \frac{1}{\tilde{a} e^{-\tilde{a}} \bar{\gamma}(-\theta,-\tilde{a})}.
	\end{equation}
	We recover the formula for $(X_t)_{t\geq0}$ by a change of time. Indeed, let $(\tilde{X}_t)_{t\geq0} := (X_{\frac{t}{q}})_{t\geq0}$. Then, $(\tilde{X}_t)_{t\geq0}$ is still a $M/M/\infty$ queue but with modified transition rates : the arrival rate is now $\tilde{a} = \frac{a}{q}$ and the serving rate is $1$. Moreover, its time of death $\tilde{T} := \inf\{t>0, \tilde{X}_t=0\}$ is just a deformation of the one of $(X_t)_{t\geq0}$ : \[ \tilde{T} = q T.\] Therefore we have \[\mathbb{E}_1[e^{\theta T }] = \mathbb{E}_1[e^{\frac{\theta}{q} \tilde{T}}],\] and we can use $\eqref{eqn:laplaceavantchangementtemps}$ to obtain $\eqref{eqn:LaplaceT}$.
\end{proof}

In the same work, Guillemin and Pinchon \cite{GuilleminPinchon98} proved that $\bar{\gamma}(\cdot, -\frac{a}{q})$ admits an unique root in $(-1,0)$ which gives the first singularity point of $\mathbb{E}[e^{\theta T }]$, namely $\theta^*$ the minimal rate of survival of $(X_t)_{t\geq0}$. 

\begin{lemma}[Characterization of $\theta^*$]\label{thm:firstcharactheta*}
	The exponential rate of survival $\theta^*$ is the unique solution in $(0,q)$ of 
	\begin{equation}\label{eqn:firstcharactheta*}
		\bar{\gamma}(-\frac{\theta^*}{q},-\frac{a}{q}) =0,\qquad\text{i.e.}\qquad \sum_{k=0}^{\infty} \frac{(\frac{a}{q})^k}{(k-\frac{\theta^*}{q})k!} = 0.
	\end{equation}
\end{lemma}

\begin{proof}
	It is a direct application of $\eqref{eqn:theta*apartirde1}$, $\eqref{eqn:LaplaceT}$ and of the fact that $\bar{\gamma}(\cdot, -\frac{a}{q})$ has a unique root in $(-1,0)$.
\end{proof}

To conclude this first part, note that the definition $\eqref{eqn:bargamma}$ of $\bar{\gamma}$ implies that for any fixed $x\in\mathbb{C}$, the function $s\mapsto\bar{\gamma}(s,x)$ is meromorphic on $\mathbb{C}$ with the negative integers as poles. Moreover, for $s,x>0$, $\bar{\gamma}(s,x)$ admits an integral representation which is linked to the lower incomplete gamma function $\gamma(s,x)$, indeed we can write

\begin{equation}\label{eqn:gammaintegral}
	\begin{split}
		\bar{\gamma}(s,x) &= \frac{1}{x^s} \sum_{k=0}^{\infty} (-1)^k \frac{x^{k+s}}{(k+s)k!}, \\
		&= \frac{1}{x^s} \int_{0}^{x} \sum_{k=0}^{\infty} (-1)^k \frac{t^{s+k-1}}{k!} dt, \\
		&=\frac{1}{x^s} \int_{0}^{x} t^{s-1} e^{-t} dt = \frac{1}{x^s} \gamma(s,x).
	\end{split}
\end{equation}

This integral form acts as a link with the work of Mart\'{\i}nez and Ycart we present in the next section.

\section{Generating function of the minimal quasi-stationary distribution}

In this section, we study $\theta^*$ using its characterization as the rate of survival of the minimal quasi-stationary distribution $\nu^*$, which can be written as follows. 

\begin{proposition}
	Let $\nu$ be a quasi-stationary distribution of $(X_t)_{t\geq0}$, then its parameter $\theta(\nu)$ satisfies : 
	\begin{equation}\label{eqn:nu*system}
		\nu Q|_{\{1,2,\dots\}} = -\theta(\nu) \nu 
		\quad\text{and}\quad
		\theta(\nu) = Q(1,0) \nu(1).
	\end{equation}
	Moreover, for the minimal quasi-stationary distribution $\nu^*$, \[\theta(\nu^*) = \theta^*.\]
\end{proposition}

This characterization is the most commonly used when computing the quasi-stationary distributions of birth and death processes : from $\eqref{eqn:nu*system}$ one can get a functional equation on the generating function of $\nu^*$. This is the method used by Mart\'{\i}nez and Ycart \cite{MartinezYcart_cutoff} to get the following result on the quasi-stationary distributions of $(X_t)_{t\geq0}$. 

\begin{proposition}[Mart\'{\i}nez, Ycart 2001]
	Let $\nu$ be a q.s.d of the process $(X_t)_{t\geq0}$ and denote by $g_\nu:s\mapsto\sum_{k=1}^{\infty}\nu(k)s^k$ its generating function. Then $g_\nu$ is given for any $s\in(0,1)$ by : 
	\begin{equation}\label{eqn:generatingnu}
		g_\nu(s) = 1 + (1-s)^{\frac{\theta(\nu)}{q}}e^{\frac{as}{q}}(-1 + \frac{a}{q}\int_{0}^{s}(1-x)^{-\frac{\theta(\nu)}{q}}e^{-\frac{ax}{q}}dx),
	\end{equation}
	where $\theta(\nu)=q\nu(1)$.
\end{proposition}

Moreover, they remarked that the function $\theta \mapsto \frac{a}{q}\int_{0}^{1}(1-x)^{-\frac{\theta}{q}}e^{-\frac{ax}{q}}dx$ is continuous and strictly increasing, with values ranging from $1 - e^{-\frac{a}{q}}$ to $+\infty$. Hence, there exists a unique $\tilde{\theta}\in(0,q)$ such that \[ \frac{a}{q}\int_{0}^{1}(1-x)^{-\frac{\tilde{\theta}}{q}}e^{-\frac{ax}{q}}dx = 1. \] Therefore, if we consider a quasi-stationary distribution $\nu$ of $(X_t)_{t\geq0}$ and suppose that $\theta(\nu)>\tilde{\theta}$, then by $\eqref{eqn:generatingnu}$ we have $g_\nu(s) > 1$ for $s$ close enough to $1$ which contradicts the fact that $g_\nu$ is a generating function of a probability measure. So \[ \theta^* \leq \tilde{\theta}. \]

Through numerical simulations, Mart\'{\i}nez and Ycart conjectured that this last inequality is in fact an equality. Thanks to the characterization of Lemma \ref{thm:firstcharactheta*} we can prove this fact. We start the proof with two technical lemmas.

Firstly, we want to understand the domain of definition of the function $F$ given by 
\begin{equation}\label{eqn:defF}
	F: (s,x) \mapsto -x \int_{0}^{1} (1-y)^s e^{xy} dy.
\end{equation}
and its properties over it.

\begin{lemma}\label{lemma:Fholo}
	Denote $\Omega := \{ s\in\mathbb{C}, \mathcal{R}e(s)>-1\}$. Then $F$ verifies : 
	\begin{itemize}
		\item for any fixed $x\in\mathbb{C}$, the function $s\mapsto F(s,x)$ is holomorphic on $\Omega$;
		\item for any fixed $s\in\Omega$, the function $x\mapsto F(s,x)$ is holomorphic on $\mathbb{C}.$
	\end{itemize}
\end{lemma}

\begin{proof}
	We use a classic result on holomorphicity for functions defined by an integral (see for example \cite[Theorem I.7. p.~308]{queffelec2020analyse} ). Let us fix $x\in\mathbb{C}$. \\
	First, for $y\in(0,1)$, the function $s\mapsto(1-y)^{s} e^{xy}$ is holomorphic on $\Omega$, as can be seen by writing it as \[(1-y)^{s} e^{xy} = e^{s\ln(1-y)+xy}.\] 
	Secondly, we show that for $s$ in any compact subdomain $K$ of $\Omega$, the function $y\mapsto(1-y)^{s} e^{xy}$ is dominated by a positive and integrable function on $(0,1)$. But, for such a set $K$, there exists $r_{\min}$ such that for any $s\in K$ \[\mathcal{R}e(s)>r_{\min} >-1.\] Therefore, for any $s\in K$ and $y\in(0,1)$ \[|(1-y)^s e^{xy}| = (1-y)^{\mathcal{R}e(s)} e^{\mathcal{R}e(x)y} \leq (1-y)^{r_{\min}} e^{\mathcal{R}e(x)y}. \] Since $r_{\min} > -1$, the function $y\mapsto(1-y)^{r_{\min}} e^{\mathcal{R}e(x)y}$ is indeed positive and integrable on $(0,1)$. \\
	This also implies that for fixed $s\in\Omega$, the function $y\mapsto(1-y)^{s} e^{xy}$ is integrable on $(0,1)$. These three assertions ensure that $s\mapsto F(s,x)$ is holomorphic on $\Omega$.
	
	For the second function, we use similar arguments. Let us fix $s\in\Omega$. For any $y\in(0,1)$, the function $x\mapsto(1-y)^{s} e^{xy}$ is holomorphic on $\mathbb{C}$. The second point is the same as before. And, finally, let $C$ be a compact subdomain of $\mathbb{C}$, then, there exists $R_{\max}>0$ such that for any $x\in C$, \[ \mathcal{R}e(x)\leq R_{\max}. \] Therefore, for any $x\in C$ and $y\in(0,1)$ \[ |(1-y)^s e^{xy}| \leq (1-y)^{\mathcal{R}e(s)} e^{R_{\max} y }, \] which concludes the proof.
\end{proof}

We remark that the expression $\eqref{eqn:defF}$ is not far from the integral representation of $\bar{\gamma}$ given by $\eqref{eqn:gammaintegral}$. In fact, we can prove the following link between $F$ and $\bar{\gamma}$.

\begin{lemma}\label{lem:Feqgamma}
	For any $s,x\in(0,\infty)$, we have the following identity : 
	\begin{equation}\label{eqn:Feqgamma}
		F(s,x) = 1 - s e^x \bar{\gamma}(s,x).
	\end{equation}
\end{lemma}

\begin{proof}
	Let $s,x>0$, taking $u=1-y$ under the integral in $\eqref{eqn:defF}$ we have: \[ F(s,x)=-x e^x \int_{0}^{1} u^s e^{-u x } du.\]  We integrate by parts and get: 
	\begin{equation*}
		\begin{split}
			F(s,x) &= -x e^x (-\frac{e^{-x}}{x} + \frac{s}{x} \int_{0}^{1} u^{s-1} e^{-u x} du), \\
			& = 1 - s e^x \int_{0}^{1} u^{s-1} e^{-u x} du.
		\end{split}
	\end{equation*}
	Finally we take $y=ux$ under the integral, it follows: \[F(s,x) = 1 - s e^x \frac{1}{x^s} \int_{0}^{x} y^{s-1} e^{-y} dy = 1 - s e^x \bar{\gamma}(s,x), \]
	using the integral form of $\bar{\gamma}$ given by $\eqref{eqn:gammaintegral}$.
\end{proof}

Recall that we want to prove that $\theta^*$ is solution of $\eqref{eqn:theta*isgamma*}$, i.e \[F(-\frac{\theta^*}{q}, -\frac{a}{q}) = 1,\] so we want to take $s=-\frac{\theta^*}{q}$ and $x= -\frac{a}{q}$ in $\eqref{eqn:Feqgamma}$ and use our first characterization of $\theta^*$ $\eqref{eqn:firstcharactheta*}$ to conclude. But both conditions $s>0$ and $x>0$ are necessary in our proof of Lemma \ref{lem:Feqgamma}. Therefore, we need to extend the identity $\eqref{eqn:Feqgamma}$ to negative value of $s$ and $x$. We do so by complex analysis arguments. 

\begin{proof}[Proof of Theorem \ref{thm:main}]
	It is enough to show that Identity $\eqref{eqn:Feqgamma}$ is verified for any $s\in\Omega$ and $x\in\mathbb{C}$, then $\eqref{eqn:theta*isgamma*}$ will follow as a corollary of Lemma \ref{thm:firstcharactheta*}. This is done in two steps. 
	
	First, let us fix $s>0$. We know from a previous remark and Lemma \ref{lemma:Fholo} that both $x\mapsto1 - se^{x} \bar{\gamma}(s,x)$ and $x\mapsto F(s, x)$ are holomorphic on $\mathbb{C}$. Moreover, we showed in Lemma \ref{lem:Feqgamma} that they are equal on the real half-line $(0,+\infty)$ so, by analycity, they are equal on the plane $\mathbb{C}$. Hence, we have for any $x\in\mathbb{C}$ that for any $s>0$ \[ F(s,x) = 1 - s e^x \bar{\gamma}(s,x) .\] Therefore, we can proceed to the second step and fix $x\in\mathbb{C}$ and use symmetric arguments to extend the last equality, namely $\eqref{eqn:Feqgamma}$, to $s\in\Omega$.
	
	Finally, since $\theta^*$ is the unique solution in $(0,q)$ of $\eqref{eqn:firstcharactheta*}$, it is also the unique solution in $(0,q)$ of $\eqref{eqn:theta*isgamma*}$, which concludes the proof. 
\end{proof}

\paragraph{Acknowledgments.}
	We acknowledge financial support by the Investissements d'Avenir programme 10-LABX-0017, Sorbonne Paris Cité, Laboratoire d'excellence INFLAMEX.

\end{document}